\documentclass[12pt]{amsart}
\usepackage{amssymb,amscd}
\usepackage{amsmath}

\begin{document}
\newtheorem{thm}{Theorem}[section]
\newtheorem{lem}[thm]{Lemma}
\newtheorem{prop}[thm]{Proposition}
\newtheorem{cor}[thm]{Corollary}
\theoremstyle{definition}
\newtheorem{ex}[thm]{Example}
\newtheorem{rem}[thm]{Remark}
\newtheorem{prob}[thm]{Problem}
\newtheorem{thmA}{Theorem}
\renewcommand{\thethmA}{}
\newtheorem{defi}[thm]{Definition}
\renewcommand{\thedefi}{}
\input amssym.def
\long\def\alert#1{\smallskip{\hskip\parindent\vrule%
\vbox{\advance\hsize-2\parindent\hrule\smallskip\parindent.4\parindent%
\narrower\noindent#1\smallskip\hrule}\vrule\hfill}\smallskip}
\def\ff{\frak}
\def\Spec{\mbox{\rm Spec}}
\def\type{\mbox{ type}}
\def\Hom{\mbox{ Hom}}
\def\rank{\mbox{ rank}}
\def\Ext{\mbox{ Ext}}
\def\Ker{\mbox{ Ker}}
\def\Max{\mbox{\rm Max}}
\def\End{\mbox{\rm End}}
\def\l{\langle\:}
\def\r{\:\rangle}
\def\Rad{\mbox{\rm Rad}}
\def\Zar{\mbox{\rm Zar}}
\def\Supp{\mbox{\rm Supp}}
\def\Rep{\mbox{\rm Rep}}
\def\cal{\mathcal}
\title[Stone and strongly complete MV-algebras]{Stone MV-algebras and Strongly complete MV-algebras}
\thanks{2010 Mathematics Subject Classification.
06D35, 06E15, 06D50\\Key words: topological MV-algebra, Stone MV-algebra, zero-dimensional, profinite MV-algebra.}
\thanks{\today}
\author{Jean B. Nganou}
\address{Department of Mathematics, University of Oregon, Eugene,
OR 97403} \email{nganou@uoregon.edu}
\begin{abstract} The article has two main objectives: characterize compact Hausdorff topological MV-algebras and Stone MV-algebras on one hand, and characterize strongly complete MV-algebras on the other hand. We obtain that compact Hausdorff topological MV-algebras are product (both topological and algebraic) of copies $[0,1]$ with the interval topology and finite \L ukasiewicz chains with discrete topology. Going one step further we also prove that Stone MV-algebras are product (both topological and algebraic) of finite \L ukasiewicz chains with discrete topology. In the second part we prove that an MV-algebra is isomorphic to its profinite completion if and only if it is profinite and its only maximal ideals of finite ranks are principal.
\vspace{0.20in}\\
{\noindent} \begin{center}\textit{To the memory of a great mentor, John V. Leahy (1937-2015).}\end{center}
\end{abstract}
\maketitle
\section{Introduction}
An MV-algebra can be defined as an Abelian monoid $(A,\oplus , 0)$ with an involution $\neg:A\to A$ (i.e., $\neg \neg x=x$ for all $x\in A$) satisfying the following axioms for all $x, y\in A$: $\neg 0\oplus x=\neg0$,
$\neg(\neg x\oplus y)\oplus y=\neg(\neg y\oplus x)\oplus x$. For any $x,y\in A$, if one writes $x\leq y$ when $\neg x\oplus y=\neg 0:=1$, then $\leq$ induces a partial order on $A$, which is in fact a distributive lattice order where $x\vee y=\neg(\neg x\oplus y)\oplus y$ and $x\wedge y=\neg(\neg x\vee \neg y)$. MV-algebras form an equational class that contains the variety of Boolean algebras (as MV-algebras satisfying $x\wedge \neg x=0$). An ideal of an MV-algebra is a nonempty subset $I$ of $A$ such that (i) for all $x,y\in I$, $x\oplus y\in I$ and (ii) for all $x\in A$ and $y\in I$, $x\leq y$ implies $x\in I$. A prime ideal of $A$ is proper ideal $P$ such that whenever $x\wedge y\in P$ with $x, y\in A$, then $x\in P$ or $y\in P$. Maximal ideal has the usual meaning. \par
A topological MV-algebra is an MV-algebra $(A,\oplus, \neg, 0)$ together with a topology $\tau$ such that $\oplus$ and $\neg$ (and in particular $\vee$, $\wedge$) are $\tau$-continuous. A detailed treatment of topological MV-algebras, a category that is clearly a subcategory of topological lattices can be found in \cite{Hoo}. A topological MV-algebra is called a Stone MV-algebra if its topology is Stone (compact, Hausdorff, and zero-dimensional). The prototype of (topological) MV-algebra is the unit real interval $[0,1]$ equipped with the operation of truncated addition $x\oplus y=\text{min}\{x+y,1\}$, negation $\neg x=1-x$, and the identity element $0$.  For each integer $n\geq 2$, $\L_n=[0,1]\cap \mathbb{Z}\dfrac{1}{n-1}$ is an MV-subalgebra of $[0,1]$, (the \L ukasiewicz chain with $n$ elements), and up to isomorphism every nontrivial finite MV-chain is of this form. Each $\L_n$ ($n\geq 2$) is clearly a Stone MV-algebra under the discrete topology. \par
A great deal of research on topological lattices has gravitated around the following important problems.
\begin{prob}\label{pb1}
An algebra is profinite if it is the inverse limit of an inverse system of finite algebras. Since the inverse limit of an inverse family of finite discrete spaces is a Stone space, each profinite algebra is naturally a Stone topological algebra. A question of interest that arises is: for what algebraic categories $\mathbb{A}$ is the converse true?, that is every Stone $\mathbb{A}$-algebra is profinite?
\end{prob}
\begin{prob}\label{pb2}
 For what algebraic categories $\mathbb{A}$ is zero-dimensionality automatic for Stone $\mathbb{A}$-algebras? In other words, what algebraic categories $\mathbb{A}$ satisfy: an $\mathbb{A}$-algebra is Stone if and only if it is compact and Hausdorff?
 \end{prob}
Most of \cite[Section VI.2.6]{Jo} is devoted to investigating problem \ref{pb1}. It is shown that the categories of Boolean algebras, of distributive lattices, of Heyting algebras all satisfy the property stated in problem \ref{pb1}. \par
As for problem \ref{pb2}, the category of Boolean algebras is most likely the best known example. Indeed, a result of Strauss \cite{DPS} states that a Boolean algebra admits a Stone topology if and only if it admits a compact and Hausdorff topology if and only if it is a power set algebra. Other well-known categories with this property include the category of orthomodular lattices \cite[Lemma 3]{CG}, and the category of Heyting algebras \cite{BH2}. A new and simpler proof of Strauss's result can be found in \cite{BH1}.\par
During the author's BLAST 2014's presentation on ``profinite MV-algebras and multisets" \cite{jbn}, several attendees asked about the position of the category of MV-algebras with respect to the above problems. The primary goal of the present paper is to provide a complete answer to those questions and study related problems. \par
With respect to problem \ref{pb1}, we obtain that the category of MV-algebras is another is another category with the stated property. Indeed, we prove that an MV-algebra admits a 
Stone topology if and only if it is profinite if and only if it is the product of finite \L ukasiewicz chains. \par
With respect to problem \ref{pb2}, since $[0,1]$ with the interval topology is clearly a compact Hausdorff topological MV-algebra (and this is the only topology making $[0,1]$ a topological MV-algebra), which is not zero-dimensional, then the category of MV-algebras does not have the property of problem \ref{pb2}. Indeed, we obtain that MV-algebras admitting compact Hausdorff topologies are product of copies $[0,1]$ and finite \L ukasiewicz chains.\par
In the final section of the article, we consider the problem of MV-algebras that are isomorphic to their profinite completions. This question has been answered for Boolean algebras and Heyting algebras. Indeed, it is known that a Boolean algebra is isomorphic to its profinite completion if and only if it is finite (see for e.g., comment after the proof of \cite[Theorem 4.6]{GG}), while a Heyting algebra $H$ is isomorphic to its profinite completion if and only if it is finitely approximable, complete, and the kernel of every finite homomorphic image of $H$ is a principal filter of $H$ \cite[Theorem 3.10]{GG}. For MV-algebras, we prove that an MV-algebra whose all maximal ideals have finite ranks is isomorphic to its profinite completion (called here strongly complete) if and only if it is finite. It is also proved that a general MV-algebra is strongly complete if and only if it is profinite and all its maximal ideals of finite ranks are principal.\par
For basic MV-algebras terminologies, the reader may refer to \cite{C2}.
\section{Compact Hausdorff topological MV-algebras}
We recall that the interval topology on any poset $(A,\leq)$ is the smallest topology for which all the sets of the form $$[a,b]:=\{x\in A:a\leq x\leq b\}$$
($a, b\in A$, $a\leq b$) are closed. It is also known that the interval topology of any complete distributive complete lattice is compact and Hausdorff (see e.g., \cite[Corollary VI.1.11]{Jo}). In particular, the interval topology of any complete chain is compact and Hausdorff, and is indeed by \cite[Theorem VII.1.9]{Jo}, the only compact Hausdorff lattice topology on the chain. In addition, compact topological MV-algebras are complete (see, e.g., \cite[Theorem 3.41]{Hoo}), and (up to isomorphism) the only complete MV-chains are $[0,1]$ and $\L_n$ for some $n\geq 2$. It follows that up to homeomorphism, the only compact Hausdorff topological MV-chains are the unit interval with the interval topology and finite MV-chains with the discrete topologies. In particular, the interval topology is the only compact Hausdorff topology on the unit interval $[0,1]$ with respect to which the functions $(x,y)\mapsto \min(x+y, 1)$ and $x\mapsto 1-x$ are continuous.
\begin{lem}\label{ct2}
Let $A=\prod_{x\in X}A_x$, where $A_x=[0,1]$ or $A_x$ is a finite MV-chain. \\
Then, the only compact Hausdorff topology on $A$ that makes $A$ a topological MV-algebra is the natural product topology. That is the product topology, where each copy of $[0,1]$ carries the interval topology, while each finite MV-chain has discrete topology.
\end{lem}
\begin{proof}
Let $\mathfrak{p}$ be the product topology on $A$, where each copy of $[0,1]$ carries the interval topology, while each finite MV-chain has discrete topology. Then, $\mathfrak{p}$ is a compact Hausdorff topology on $A$ making it a topological MV-algebra (in particular a topological lattice). It follows from \cite[Theorem VII.1.9]{Jo} that $\mathfrak{p}$ is the only compact Hausdorff topology on $A$ that makes $A$ a topological MV-algebra.
\end{proof}
Note that if $A=\prod_{x\in X}A_x$, where $A_x=[0,1]$ or $A_x$ is a finite MV-chain, then $A$ is complete and completely distributive. Therefore, the interval topology on $A$ is a compact Hausdorff topology on $A$ that makes $A$ a topological MV-algebra. It follows from the Lemma that the product topology on $A$ coincides with the interval topology. \par
Next, we characterize all compact Hausdorff MV-algebras.
\begin{thm}\label{CH}
For every MV-algebra $A$, the following assertions are equivalent:
\begin{itemize}
\item[1.] $A$ admits a topology making it a compact and Hausdorff topological MV-algebra;
\item[2.] $A$ is complete and completely distributive;
\item[3.] $A$ is isomorphic to $\prod_{x\in X}A_x$, where $A_x=[0,1]$ or $A_x$ is a finite MV-chain. 
\item[4.] $A$ is isomorphic to a closed MV-subalgebra of an MV-algebra of the form $[0,1]^X$. 
\end{itemize}
\end{thm}
\begin{proof}
$1.\Rightarrow 2.$ Suppose that $A$ admits a topology making it a compact and Hausdorff topological MV-algebra. Consider the Boolean center $B(A)$ of $A$ equipped with the subspace topology of $A$. Recall that $B(A)=\{a\in A:a\wedge \neg a=0\}=f^{-1}(\{0\})$, where $f:A\to A$ is the map defined by $f(a)=a\wedge \neg a$. Since $f$ is clearly continuous and $\{0\}$ is closed, then $B(A)$ is a closed subspace of the compact space $A$. Therefore, $B(A)$ admits a compact and Hausdorff topological making it a topological Boolean algebra. Consequently, $B(A)$ is a power set algebra as proved by Strauss \cite{DPS}, in particular, $B(A)$ is atomic. In addition, since compact Hausdorff MV-algebras are complete (see e.g., \cite[Theorem 3.41]{Hoo}), then $A$ is complete and by \cite[Proposition 6.8.1]{C2} or \cite[Theorem 2.6]{Ci}, $A$ is complete and completely distributive.\\
$2.\Rightarrow 3.$ Suppose that $A$ is complete and completely distributive. Then, by \cite[Proposition 6.8.1]{C2} $A$ is the direct product of complete MV-chains. But (up to isomorphism), the only complete MV-chains are finite MV-chain and $[0,1]$.\\
$3.\Rightarrow 4.$ Let $S=:\prod_{x\in X}A_x$, where $A_x=[0,1]$ or $A_x$ is a finite MV-chain. Then $A$ is isomorphic to $S$, which is  clearly a closed MV-subalgebra of an MV-algebra of the form $[0,1]^X$. \\
$4.\Rightarrow 1.$ $A$ is isomorphic to a closed MV-subalgebra $S$ of an MV-algebra of the form $[0,1]^X$. Then, with the subspace topology, $S$ is a compact Hausdorff topological MV-algebra. Consider the topology  
\end{proof}
It follows from Theorem \ref{CH} that any two compact Hausdorff topological MV-algebras that are isomorphic are automatically homeomorphic.\\
When zero-dimensionality is added to the characterization of Theorem \ref{CH}, one obtains a simple characterization of Stone MV-algebras. First, we observe that $[0,1]\times \L_2$ is not a Stone MV-algebra. Indeed, $[0,1]\times \{0\}$ is a connected subset of $[0,1]\times \L_2$, which is not a singleton. Therefore, $[0,1]\times \L_2$ is not totally disconnected.
\begin{thm}
For every MV-algebra $A$, the following assertions are equivalent:
\begin{itemize}
\item[1.] $A$ admits a topology making it a Stone topological MV-algebra;
\item[2.] $A$ admits a topology making it a Stone topological lattice;
\item[3.] $A$ is isomorphic to $\prod_{x\in X}A_x$, where each $A_x$ is a finite MV-chain;
\item[4.] $A$ admits a topology making it a compact Hausdorff topological MV-algebra and $[0,a]$ is a finite MV-chain for all atoms $a$ of $B(A)$;
\item[5.] $A$ is profinite.
\end{itemize}
\end{thm}
\begin{proof}
$1.\Rightarrow 2.$ This is clear. \\
$2.\Rightarrow 3.$ Suppose that $A$ admits a topology $\tau$ making it a Stone topological lattice. Then as in the proof of Theorem \ref{CH}, we obtain that the Boolean center $B(A)$ of $A$ admits a topology (the subspace topology) that makes it a Stone topological lattice. It follows from \cite[Proposition VII.1.16]{Jo} that $B(A)$ is complete and atomic. Since $A$ is compact, then $A$ is complete. Then, by \cite[Proposition 6.8.1]{C2} $A$ is the direct product of complete MV-chains. But (up to isomorphism), the only complete MV-chains are finite MV-chain and $[0,1]$. Hence, $A$ is isomorphic to $\prod_{x\in X}A_x$, where $A_x=[0,1]$ or $A_x$ is a finite MV-chain. We claim that there is no $x\in X$ such that $A_x=[0,1]$. First, since $\tau$ is a compact and Hausdorff topology on $A$, then by Lemma \ref{ct2} the topology of $\tau$ is the natural product topology, which is the same as the interval topology. Now, if there were an $x_0\in X$ such that $A_{x_0}=[0,1]$, then $A$ would contain a closed MV-subalgebra isomorphic to $[0,1]\times \L_2$. Indeed, let $I=\{f\in A:f(x)=0, \text{for all}\; x\ne x_0\}$, then $I$ is clearly a closed ideal of $A$, and $S:=I\cup \neg I$ is the MV-subalgebra of $A$ generated by $I$ (see e.g., \cite[Lem. 2.7.5]{DL}). Finally, $S$ is clearly closed ($\neg$ is closed map) and is isomorphic to $[0,1]\times \L_2$. As closed MV-subalgebra of a Stone MV-algebra, $S$ would be a Stone MV-algebra, which would imply that $[0,1]\times \L_2$ is a Stone MV-algebra. This would contradict the comments preceding the theorem. Thus, each $A_x$ is a finite MV-chain and  $A$ is isomorphic to $\prod_{x\in X}A_x$.\\
$3.\Rightarrow 4.$ Suppose $A$ is isomorphic to $\prod_{x\in X}A_x$, where each $A_x$ is a finite MV-chain. Then equipping each $A_x$ with the discrete topology, and considering the product topology on $A$, $A$ is clearly compact Hausdorff topological MV-algebra. In addition $B(A)\cong 2^X$ and it follows that for every atom $a$ of $B(A)$, $[0,a]\cong A_x$, which is a finite MV-chain.\\
$4.\Rightarrow 3.$ Suppose that $A$ admits a topology making it a compact Hausdorff topological MV-algebra and $[0,a]$ is a finite MV-chain for all atoms $a$ of $B(A)$. By Theorem \ref{CH}, $A$ is isomorphic to $\prod_{x\in X}A_x$, where $A_x=[0,1]$ or $A_x$ is a finite MV-chain. For each $x\in X$, consider the element $\alpha_x\in A$ defined by $\alpha_x(t)=0$ if $t\ne x$ and $\alpha_x(x)=1$. Then for each $x\in X$, $\alpha_x\in B(A)$ and $[0,\alpha_x]\cong A_x$. Therefore, $A_x$ is a finite MV-chain, for all $x\in X$.\\
$3.\Leftrightarrow 5.$  This is part of \cite[Theorem 2.5]{jbn}.  Finally, profinite algebras are always Stone algebras, hence $5.\Rightarrow 1.$
\end{proof}
Note that if $A$ is a profinite MV-algebra, then $A\cong \prod_{x\in X}[0,x]$, where $X$ is the set of atoms of $B(A)$.
\section{Strongly complete MV-algebras}
An MV-algebra will be called \textit{strongly complete}\footnote{This terminology is borrowed from the well-known theory of profinite groups.} if it is isomorphic to its profinite completion. We start by recalling the construction of the profinite completion for MV-algebras. Given an MV-algebra $A$ and an ideal $I$ of $A$, the relation $\equiv_I$ defined on $A$ by $x\equiv_I y$ if and only if $\neg(\neg x\oplus y)\oplus \neg(x\oplus \neg y)\in I$, is a congruence on $A$. The quotient MV-algebra induced by this congruence is denoted by $A/I$, and the congruence class of $a\in A$ denoted by $[a]_I$. Let $A$ be an MV-algebra and let $\text{id}_f(A)$ be the set of all ideals $I$ of $A$ such that $A/I$ is finite. For every $I, J\in \text{id}_f(A)$ such that $I\subseteq J$, let $\phi_{JI}:A/I\to A/J$ be the natural homomorphism, i.e., $\phi_{JI}([a]_I)=[a]_J$ for all $a\in A$. Observe that $(\text{id}_f(A), \supseteq)$ is a directed set, and $\left\{\text{id}_f(A), \{A/I\},\{\phi_{JI}\}\right\}$ is an inverse system of MV-algebras. The inverse limit of this inverse system is called the profinite completion of the MV-algebra $A$, and commonly denoted by $\widehat{A}$. The following description of $\widehat{A}$ is also well known:
$$\widehat{A}\cong \left\{\alpha\in \prod_{I\in \text{id}_f(A)}A/I: \phi_{JI}(\alpha(I))=\alpha(J) \; \text{whenever}\; I \subseteq J\right\} $$
In the sequel, the set of prime ideals of $A$ will be denoted by $\text{Spec}(A)$ and is endowed with the Zariski's topology. The set $\text{Max}(A)$ of maximal ideals of $A$ inherits the subspace topology of $\text{Spec}(A)$.\\
We have the following characterization of $\text{id}_f(A)$.
\begin{prop}\label{frank}
Let $A$ be an MV-algebra and $I$ be any ideal of $A$. Then $I\in \text{id}_f(A)$ if and only if there exist $M_1, M_2,\ldots, M_r\in \text{Max}(A)\cap \text{id}_f(A)$ such that $I=M_1\cap M_2\cap \ldots \cap M_r$.
\end{prop}
\begin{proof}
Suppose that $I$ is an ideal of $A$, with $A/I$ finite. Then, by \cite[Proposition 3.6.5]{C2}, there is an isomorphism $\varphi:A/I\to \prod_{i=1}^r\L_{n_i}$, for some integers $n_1, n_2, \ldots, n_r\geq 2$. For each $k=1, 2, \ldots, r$, let $M_k=\ker (q_k\circ \varphi\circ p_I)$, where $q_k$ is the natural projection $\prod_{i=1}^r\L_{n_i}\to \L_k$ and $p_I:A\to A/I$ is the canonical projection. Then $M_k$ is a maximal ideal of $A$ since $A/M_k\cong \L_k$, which is simple. In addition, it is clear that $I=M_1\cap M_2\cap \ldots \cap M_r$.\\
Conversely, suppose there exist $M_1, M_2,\ldots, M_r\in \text{Max}(A)\cap \text{id}_f(A)$ such that $I=M_1\cap M_2\cap \ldots \cap M_r$. Consider the homomorphism $\varphi:A\to \prod_{i=1}^rA/M_i$ defined by $\varphi(a)=([a]_{M_i})_i$. Then, $\ker \varphi=I$ and $A/I$ is isomorphic to a sub-MV-algebra of $\prod_{i=1}^rA/M_i$. Hence, $I\in \text{id}_f(A)$.
\end{proof}
Recall \cite{CM} that a maximal ideal $M$ of an MV-algebra $A$ is said to have a finite rank $n$, for some integer $n\geq 2$ if $A/M\cong \L_n$, otherwise one says that $M$ has infinite rank. One should observe that every maximal ideal of a Boolean algebra has finite rank. An MV-algebra $A$ is called regular \cite{BNS} if for every prime ideal $N$ of its Boolean center $B(A)$, the ideal of $A$ generated by $N$ is a prime ideal of $A$. It is known \cite[Proposition 26]{BNS} that if $A$ is a regular MV-algebra, then the map $P\mapsto P\cap B(A)$ defines a homeomorphism from $\text{Max}(A)$ onto $\text{Max}(B(A))=\text{Spec}(B(A))$.
\begin{thm}\label{main3}
Let $A$ be a regular MV-algebra in which every maximal ideal has finite rank. 
\begin{itemize}
\item[1.] Then the map $\Psi: \text{id}_f(A)\to \text{id}_f(B(A))$ defined by $\Psi(I)=I\cap B(A)$ is an inclusion preserving and reversing bijection.
\item[2.] For every $I\in \text{id}_f(A)$, the map $\Theta_I:B(A)/\Psi(I)\to B(A/I)$ defined by $\Theta_I([a])=[a]_I$ is an isomorphism.
\item[3.] For every $I\subseteq J$ in $\text{id}_f(A)$, if $\phi_{JI}:A/I\to A/J$ is the transition homomorphism, the following diagram is commutative.
$$
\begin{CD}
B(A)/\Psi(I) @>\Theta_I>> B(A/I)\\
@V\phi_{\Psi(J)\Psi(I)}VV @VVB(\phi_{JI})V\\
B(A)/\Psi(J) @>\Theta_J>> B(A/J)
\end{CD}
$$

\end{itemize}
\end{thm}
\begin{proof} 1. Since every maximal ideal of $A$ (resp. $B(A)$) has finite rank, it follows from Proposition \ref{frank} elements of $\text{id}_f(A)$ (resp. $\text{id}_f(B(A))$) are simply finite intersections of maximal ideals of $A$ (resp. $B(A)$). Therefore, from the observation before the Proposition, we obtain that $\Psi$ is well-defined. To prove that $\Psi$ is bijective, we construct its inverse. Indeed, consider $\Phi: \text{id}_f(B(A))\to \text{id}_f(A)$ defined as follows. Let $V\in \text{id}_f(A)$, then $V=V_1\cap\ldots \cap V_s$, where each $S_i$ is a maximal ideal of $B(A)$. By \cite[Proposition 26]{BNS}, there exists a unique set of maximal ideals $M_1, \ldots, M_s$ of $A$ such that $V_i=M_i\cap B(A)$ for each $i=1, \ldots, s$. Consider $I=M_1\cap\ldots \cap M_s$, which is in $\text{id}_f(A)$. It is clear that $\Psi$ and $\Phi$ are inverses of each other. Clearly $\Psi$ preserves the inclusion. For the inclusion reversing, suppose that $I\cap B(A)\subseteq J\cap B(A)$, then $I\cap J\cap B(A)=I\cap B(A)$. But, since $\Psi$ is one-to-one, then $I\cap J=I$ and $I\subseteq J$.\\
2. It is clear that the map $a\mapsto [a]_I$ is a homomorphism from $B(A)$ to $B(A/I)$, whose kernel is $\Psi(I)$. Therefore, $\Theta_I$ is an injective homomorphism. But, since $B(A)/\Psi(I)$ and $B(A/I)$ are finite Boolean algebras, then $\Theta_I$ is an isomorphism as claimed.\\
3. This follows from the definitions of the various homomorphisms.
\end{proof}
It is worth pointing out that the class of MV-algebras covered by Theorem \ref{main3} is quite large. Indeed, we have the following well-known two subclasses:
\begin{itemize}
\item[(a)] Let $B$ be a profinite Boolean algebra. Then $B$ is a regular MV-algebra in which every maximal ideal has rank $2$. This is special case of Proposition \ref{bounded} below.
\item[(b)] Every non-simple MV-chain of finite rank (see \cite[Section 8.3]{C2}) is a regular MV-algebra in which every maximal ideal has finite rank. In fact, any such MV-chain $C$ has a unique maximal ideal whose rank is the same as that of $C$.
\end{itemize}
\begin{prop}\label{Boo}
Let $A$ be a regular MV-algebra in which every maximal ideal has finite rank. Then, $$B(\widehat{A})\cong \widehat{B(A)}$$
\end{prop}
\begin{proof}We have the following sequence of isomorphisms. The first one is from the standard description of the profinite completion. The second follows from Theorem \ref{main3}(1). The third follows from Theorem \ref{main3}(2-3) and the remaining isomorphisms follow from the definition of the Boolean center.
$$
\begin{aligned}
\widehat{B(A)}&\cong \left\{\alpha\in \prod_{V\in \text{id}_f(B(A))}B(A)/V: \phi_{UV}(\alpha(V))=\alpha(U) \; \text{whenever}\; V \subseteq U\right\}\\
&\cong \left\{\alpha\in \prod_{\Psi(I)\in \text{id}_f(B(A))}B(A)/\Psi(I): \phi_{\Psi(J)\Psi(I)}(\alpha(\Psi(I)))=\alpha(\Psi(J)) \; \text{whenever}\; I \subseteq J\right\}\\
&\cong \left\{\alpha\in \prod_{I\in \text{id}_f(A)}B(A/I): \phi_{JI}(\alpha(I))=\alpha(J) \; \text{whenever}\; I \subseteq J\right\}
\end{aligned}
$$
$$
\begin{aligned}
&\cong \left\{\alpha\in B\left(\prod_{I\in \text{id}_f(A)}A/I\right): \phi_{JI}(\alpha(I))=\alpha(J) \; \text{whenever}\; I \subseteq J\right\}\\
&\cong B\left( \left\{\alpha\in \prod_{I\in \text{id}_f(A)}A/I: \phi_{JI}(\alpha(I))=\alpha(J) \; \text{whenever}\; I \subseteq J\right\}\right)\\
&\cong  B(\widehat{A})
\end{aligned}
$$
\end{proof}
\begin{cor}\label{finiterank-str}
Let $A$ be an MV-algebra in which every maximal ideal has finite rank. Then $A$ is strongly complete if and only if $A$ is finite.
\end{cor}
\begin{proof}
It is clear that finite MV-algebras are strongly complete. \\
Conversely, suppose that $A$ is strongly complete. Then $\widehat{A}\cong A$. Then, since $\widehat{A}$ is profinite, so is $A$. By \cite[Theorem 2.5]{jbn}, $A\cong \prod_{x\in X}\L_{n_x}$ for some set $X$, and some integers $n_x\geq 2$. Thus, by \cite[Proposition 36, 37]{BNS}, $A$ is regular. Hence, by Proposition \ref{Boo}, $\widehat{B(A)}\cong B(\widehat{A})\cong B(A)$. Hence, the Boolean center $B(A)$ of $A$ is strongly complete. But the only strongly complete Boolean algebras are finite ones. Thus, $B(A)\cong 2^X$ is finite. Therefore, $X$ is finite and $A$ is finite as needed.
\end{proof}
\begin{prop}\label{strong1}
Let $A$ be a strongly complete MV-algebra. Then $A$ is profinite and every maximal ideal of $A$ of finite rank is principal.
\end{prop}
\begin{proof}
Suppose that $A$ is an MV-algebra isomorphic to its profinite completion $\widehat{A}$. Then, $A$ is profinite since $\widehat{A}$ is profinite by definition. Let $M$ be a maximal of $A$ of finite rank $n$. Then, $M$ can be viewed as an ideal of $\widehat{A}$ of finite rank. Since $\widehat{A}=\left\{\alpha\in \prod_{I\in \text{id}_f(A)}A/I: \phi_{JI}(\alpha(I))=\alpha(J) \; \text{whenever}\; I \subseteq J\right\} $, then $M$ (viewed as an ideal of $\widehat{A}$) is equal to $A_M:=\{\alpha\in \widehat{A}: \alpha(M)=M\}=\{\alpha\in \prod_{I\in \text{id}_f(A)}A/I: \alpha(M)=M\}\cap \widehat{A}=\ker p_M$, where $p_M$ is the natural projection $\widehat{A}\to A/M$. But, $A_M$ is a maximal ideal of $\widehat{A}$ since $\widehat{A}/A_M\cong A/M$, which is simple. In addition, $A_M$ is clearly clopen in $\widehat{A}$ under the induced topology of $\prod_{I\in \text{id}_f(A)}A/I$. Hence, by \cite[Lemma 3.1]{jbn} $A_M$ is principal in $\widehat{A}$, and $M$ is principal in $A$.
\end{proof}
Recall \cite[Theorem 2.5]{jbn} that if $A$ is a profinite MV-algebra, then $A\cong \prod_{x\in X}\L_{n_x}$ for some set $X$ and $n_x\geq 2$ are integers. Let $\mathcal{P}_f(X)$ be the set of finite subsets of $X$. For each $S\in \mathcal{P}_f(X)$, let $A_S:=\prod_{x\in S}\L_{n_x}$. Then for $S\subseteq T$ in $\mathcal{P}_f(X)$, there exists a natural homomorphism $\phi_{ST}: A_T\to A_S$ defined by $\phi_{ST}(\alpha)(x)=\alpha(x)$ for all $\alpha\in A_T$ and $x\in S$. It is easily verified that $(\mathcal{P}_f(X), \supseteq)$ is a directed set and $\left\{ \{\mathcal{P}_f(X)\}, \left\{A_S\right\}_S, \{\phi_{ST}\}\right\}$ is an inverse system. Let $$\widetilde{A}:=\varprojlim\left\{ \{\mathcal{P}_f(X)\}, \left\{A_S\right\}_S, \{\phi_{ST}\}\right\}$$
\begin{lem}\label{hat=tilde}
Let $A$ be a profinite MV-algebra such that every maximal ideal of $A$ of finite rank is principal. Then, $$\widehat{A}\cong \widetilde{A}$$
\end{lem}
\begin{proof}
It is enough to prove that the inverse systems $\left\{\text{id}_f(A), \{A/I\},\{\phi_{JI}\}\right\}$ and $\left\{ \{\mathcal{P}_f(X)\}, \left\{A_S\right\}_S, \{\phi_{ST}\}\right\}$ are isomorphic. Consider the map $\Theta: \mathcal{P}_f(X) \to \text{id}_f(A)$ defined by $\Theta(S)=\cap_{x\in S}M_x$, where $M_x$ is the kernel of the natural projection $p_x:A\to \L_{n_x}$. It follows from Proposition \ref{frank}, the fact every maximal ideal of $A$ of finite rank is principal and \cite[Proposition 3.1]{jbn} that $\Theta$ is a inclusion-preserving bijection. In addition, for all $S\in \mathcal{P}_f(X)$, there is a canonical isomorphism $\varphi_S: A_S\to A/\Theta(S)$. To see that the systems $\left\{\text{id}_f(A), \{A/I\},\{\phi_{JI}\}\right\}$ and $\left\{ \{\mathcal{P}_f(X)\}, \left\{A_S\right\}_S, \{\phi_{ST}\}\right\}$ are isomorphic, one needs to verify that for all $S\subseteq T$, the following diagram is commutative.
$$
\begin{CD}
A_T @>\varphi_T>> A/\Theta(T)\\
@V\phi_{ST}VV @VV\phi_{\Theta(S)\Theta(T)}V\\
A_S @>\varphi_S>> A/\Theta(S)
\end{CD}
$$

But, this is a simple verification.
\end{proof}
\begin{thm}\label{strong2}
Let $A$ be a profinite MV-algebra such that every maximal ideal of $A$ of finite rank is principal. Then $A$ is strongly complete .
\end{thm}
\begin{proof} Recall that $\widetilde{A}\cong \left\{ f\in \prod_{S\in \mathcal{P}_f(X)}A_S: \phi_{ST}(f(T))=f(S), \; \text{whenever}\; S\subseteq T\right\}$.
Since $\widehat{A}\cong \widetilde{A}$ from Lemma \ref{hat=tilde}, we shall prove that $A$ is isomorphic to $\widetilde{A}$. For $f\in A$, define $\varphi(f)\in \prod_{S\in \mathcal{P}_f(X)}A_S$ by $\varphi(f)(S)(x)=f(x)$ for all $S\in \mathcal{P}_f(X)$ and all $x\in S$. We claim that $\varphi(f)\in \widetilde{A}$. Suppose $S\subseteq T$, and $x\in S$, then $\phi_{ST}(\varphi(f)(T))(x)=\varphi(f)(T)(x)=f(x)=\varphi(f)(S)(x)$. Therefore, $\varphi(f)\in \widetilde{A}$ and $\varphi$ is a well-defined map from $A\to \widetilde{A}$. It is easy to verify that $\varphi(f\oplus g)=\varphi(f)\oplus \varphi(g)$ and $\varphi(\neg f)=\neg \varphi(f)$, and so $\varphi$ is a homomorphism. It remains to prove that $\varphi$ is bijective.\\
\underline{Injectivity:} Let $f, g\in A$ such that $\varphi(f)=\varphi(g)$, then for all $x\in X$, $\varphi(f)(\{x\})(x)=\varphi(g)(\{x\})(x)$. Hence, $f(x)=g(x)$ for all $x\in X$, and $f=g$. Thus, $\varphi$ is injective.\\
\underline{Surjectivity:} Let $\alpha\in \widetilde{A}$, then $\alpha(\{x\})\in A_{\{x\}}$ (=$\prod_{t\in \{x\}}\L_{n_t}$). Define $f\in A$ by $f(x)=\alpha(\{x\})(x)$. Then, for every $S\in\mathcal{P}_f(X)$ and every $x\in S$, $\varphi(f)(S)(x)=f(x)=\alpha(\{x\})(x)=\phi_{\{x\}S}(\alpha(S))(x)=\alpha(S)(x)$. Hence, $\varphi(f)=\alpha$, and $\varphi$ is surjective.\\
Thus, $\varphi$ is an isomorphism as desired.
\end{proof}
Combining Proposition \ref{strong1} and Theorem \ref{strong2}, we obtain a necessary and sufficient condition for the strong completeness of MV-algebras.
\begin{cor}\label{iffstrong}
An MV-algebra $A$ is strongly complete if and only if $A$ is profinite and every finite rank maximal ideal of $A$ is principal.
\end{cor}
Since every maximal ideal of a Boolean algebra has rank $2$, and every infinite profinite Boolean algebra (powerset) has non-principal maximal ideals, then it follows from Corollary \ref{iffstrong} that the only strongly complete Boolean algebras are the finite ones. This provides another proof of this well-known result.\par
\section{Examples}
In this section we construct concrete examples and non-examples of strongly complete MV-algebras. The main tool used here is the well-known connection between the maximal spectrum of profinite and the theory of ultrafilters. Recall that if $A$ is a profinite MV-algebra, then $A\cong \prod_{x\in X}\L_{n_x}$ for some set $X$, and some integers $n_x\geq 2$. In addition, by \cite[Corollary 3.4]{jbn}, the set $\eta(A):=\{n_x:x\in X\}$ is uniquely determined by $A$. \par 
In the remainder of this section $A:=\prod_{x\in X}\L_{n_x}$. We shall recall a known description of the maximal spectral space of $A$. It is likely that the presentation given of this description is not new, but for lack of direct reference, and also to increase the readability, we include it.
For $f\in A$ and $\epsilon>0$, $$D(f, \epsilon):=\{x\in X:f(x)<\epsilon$$
It is easy to see that for all $f, g\in A$ and $\epsilon, \delta>0$:
$D(f\oplus g, \min(\epsilon, \delta))\subseteq D(f, \epsilon)\cap D(g, \delta)$
\begin{defi}
Let $f\in A$, $\mathcal{U}$ an ultrafilter of $X$ and $t\in [0,1]$. We say that $f$ converges to $t$ along $\mathcal{U}$ if for every $\epsilon>0$, there exists $S\in \mathcal{U}$ such that $f(S)\subseteq (t-\epsilon, t+\epsilon)$.
\end{defi}
Note that this is equivalent to for every $\epsilon>0$, $\{x\in X: |f(x)-t|<\epsilon\}=f^{-1}(t-\epsilon, t+\epsilon)\in \mathcal{U}$.\\
Since $[0,1]$ is compact and Hausdorff, then every $f\in A$ has a unique limit along any ultrafilter $\mathcal{U}$ of $X$ (see for e.g., \cite[Prop. 1.7]{IGG}). We denote this limit by $$\lim _\mathcal{U}f$$
Note that $f$ converges to $0$ along $\mathcal{U}$ if and only if  for every $\epsilon>0$, $D(f,\epsilon)\in \mathcal{U}$. \\
For every ultrafilter $\mathcal{U}$ of $X$, let $$M_{\mathcal{U}}:=\{f\in A: \lim _\mathcal{U}f=0\}=\{f\in A: D(f,\epsilon)\in \mathcal{U}\; \text{for all}\; \epsilon>0\}$$
For every maximal ideal $M$ of $A$, let $$\mathcal{U}_M:=\{S\subseteq X:D(f,\epsilon)\subseteq S\; \text{for some}\; f\in M\; \text{and}\; \epsilon>0\}$$
\begin{prop}\label{alglimit}
For every ultrafilter $\mathcal{U}$ of $X$, define the map $\Psi_{\mathcal{U}}:A\to [0,1]$ by $\Psi_{\mathcal{U}}(f)=\lim _\mathcal{U}f$. Then $\Psi$ is an MV-homomorphism whose kernel is $M_{\mathcal{U}}$.
\end{prop}
\begin{proof}
Let $f, g\in A$, and let $\lim _\mathcal{U}f=t$, $\lim _\mathcal{U}g=s$. We need to show that $\lim _\mathcal{U}(f\oplus g)=t\oplus s$. Let $\epsilon>0$, then since $\oplus:[0,1]\times [0,1]\to [0,1]$ is continuous, there exists $\delta>0$ such that for all $a,b\in[0,1]$ ($|a-t|<\delta$ and $|b-s|<\delta$ $\Rightarrow$ $|a\oplus b-t\oplus s|<\epsilon$). On the other hand, since $\lim _\mathcal{U}f=t$ and $\lim _\mathcal{U}g=s$, then $S:=\{x\in X:|f(x)-t|<\delta \}\in \mathcal{U}$ and $T:=\{x\in X:|g(x)-s|<\delta\}\in \mathcal{U}$. But, $S\cap T\subseteq \{x\in X:|f(x)\oplus g(x)-t\oplus s|<\epsilon\}$ and $S\cap T\in \mathcal{U}$, then $\{x\in X:|f(x)\oplus g(x)-t\oplus s|<\epsilon\}\in \mathcal{U}$. Thus, $f\oplus g$ converges to $t\oplus s$ along $\mathcal{U}$. Using the fact that $\neg:[0,1]\to [0,1]$ is continuous, one can prove in a  similar manner that $\lim _\mathcal{U}\neg f= \neg \lim _\mathcal{U}f$. Hence, $\Psi_{\mathcal{U}}(\neg f)=\neg \Psi_{\mathcal{U}}(f)$ and $\Psi_{\mathcal{U}}(f\oplus g)=\Psi_{\mathcal{U}}(f)\oplus \Psi_{\mathcal{U}}(g)$, for all $f, g\in A$. Therefore, $\Psi$ is an MV-homomorphism. Clearly, from the definition $\ker \Psi_{\mathcal{U}}=M_{\mathcal{U}}$.
\end{proof}
The set of maximal ideals of $A$ (in fact any MV-algebra) carries a natural Zariski topology. The resulting space is known as the maximal spectral space of $A$ (see e.g., \cite[Chap. 4]{Mu}). Since $A=\prod_{x\in X}\L_{n_x}$ is a profinite MV-algebra, it is known that the maximal spectral space of $A$ is the Stone \v{C}ech compactification $\beta X$ of $X$ endowed with the discrete topology. More precisely, the following result holds.
\begin{thm}
The map $M\mapsto \mathcal{U}_M$ is a homeomorphism from $Max(A)\to \beta X$ whose inverse is $\mathcal{U}\mapsto M_{\mathcal{U}}$.
\end{thm}
Note that it follows that every maximal ideals of $A$ are of the form $M_{\mathcal{U}}$, where $\mathcal{U}$ is an ultrafilter of $X$. In addition, $M_{\mathcal{U}}$ is a principal maximal ideal of $A$ if and only if $\mathcal{U}$ is a principal ultrafilter of $X$. Moreover, note that from Proposition \ref{alglimit}, the quotient MV-algebra $A/M_{\mathcal{U}}$ is naturally isomorphic to the sub-MV-algebra of $[0,1]$ of all $\mathcal{U}$-limits of elements of $A$. We can apply this and obtain the following result.
\begin{prop} \label{bounded}
Let $A$ be a profinite MV-algebra such that $\eta(A)$ is bounded. Then every maximal ideal of $A$ has finite rank.
\end{prop}
\begin{proof}
Suppose that $\{n_x:x\in X\}$ has a maximum element, say $n$ and let $M$ be a maximal ideal of $A$. Then, $M=M_{\mathcal{U}}$ for some ultrafilter $\mathcal{U}$ of $X$. We prove that Im$\Psi_{\mathcal{U}}$ is finite. It is enough to show that for every nonzero $t\in \text{Im}\Psi_{\mathcal{U}}$, $nt=1$. So, let $f\in A$ such that $\lim _\mathcal{U}f\ne 0$, there there exists $\delta_1>0$ such that $D(f, \delta_1)\notin \mathcal{U}$. Hence, $X\setminus D(f, \delta_1)\in \mathcal{U}$. We need to show that $\lim _\mathcal{U}nf=1$. But contradiction, suppose that $\lim _\mathcal{U}nf\ne 1$, then there exists $\delta_2>0$ such that $\{x\in X:|nf(x)-1|<\delta_2\}\notin \mathcal{U}$. In particular, note that $\delta_2<1$. In addition, since for all $x\in X$, $f(x)=0$ or $nf(x)=1$, then $\{x\in X:|nf(x)-1|<\delta_2\}=\{x\in X:nf(x)=1\}$. Thus, $Z(f):=\{x\in X:f(x)=0\}\in \mathcal{U}$, since $Z(f)=X\setminus \{x\in X:nf(x)=1\}$. Therefore, $\emptyset=Z(f)\cap (X\setminus D(f, \delta_1))\in \mathcal{U}$, which is a contradiction.

\end{proof}
It follows from Corollary \ref{finiterank-str} that the only profinite MV-algebras $A$ with bounded $\eta(A)$ are the finite ones. \par
The next examples illustrate the situation when $\eta(A)$ is unbounded.
Unlike Boolean algebras, there exists infinite strongly complete MV-algebras.
\begin{ex}
Let $A=\prod_{n=1}^{\infty}\L_{n+1}$. We claim that every non-principal maximal ideal of $A$ has infinite rank. Let $M$ be a non-principal maximal ideal of $A$, then $M=M_{\mathcal{U}}$ for some free ultrafilter $\mathcal{U}$ of $\mathbb{N}$. An application of the Archimedean property of real numbers gives the following fact: For every $0\leq a<b\leq 1$, there exists a natural number $N$ such that for all integers $n\geq N$, there exists an integer $k_n$ satisfying $a\leq k_n/n\leq b$. It follows that for every $0\leq a<b\leq 1$, there exists a natural number $N$ and $f\in A$ such that $a\leq f(n)\leq b$ for all integers $n\geq N$. Since $a\leq f(n)\leq b$ for all integers $n\geq N$, then $a\leq \lim _\mathcal{U}f\leq b$ as $\mathcal{U}$ is a free ultrafilter. Thus, Im$\Psi_{\mathcal{U}}$ is dense in $[0,1]$, and in particular $A/M$ is infinite. Note that indeed $A/M\cong [0,1]$ (see for e.g.,\cite{Mu} ).
\end{ex}
Not every profinite MV-algebra $A$ with $\eta(A)$ unbounded is strongly complete as the next example shows.
\begin{ex}
Let $A=\prod_{k=1}^{\infty}A_k$, where $A_{2k-1}=\L_2$ and $A_{2k}=\L_{2k}$. Let $S$ be the set of odd natural numbers and $\mathcal{U}$ be any free ultrafilter of $\mathbb{N}$ containing $S$. We claim that $A/M_{\mathcal{U}}\cong \L_2$. In fact, suppose that $f\notin M_{\mathcal{U}}$, then there exists $\delta>0$ such that $D(f,\delta)\notin \mathcal{U}$. Hence $\mathbb{N}\setminus D(f,\delta)\in \mathcal{U}$ and $S\cap (\mathbb{N}\setminus D(f,\delta))\in \mathcal{U}$. But, $S\cap (\mathbb{N}\setminus D(f,\delta))=\{n\in \mathbb{N}:f(n)=1\}\subseteq \{n\in \mathbb{N}:|f(n)-1|<\epsilon \}$ for every $\epsilon>0$. Thus, $\{n\in \mathbb{N}:|f(n)-1|<\epsilon \}\in \mathcal{U}$ for every $\epsilon>0$. Hence $ \lim _\mathcal{U}f=1$, and Im$\Psi_{\mathcal{U}}=\L_2$. Therefore, $M_{\mathcal{U}}$ is a non-principal maximal ideal of $A$ of finite rank. Whence, $A$ is not strongly complete. 
\end{ex}
\section{Conclusion and final remarks}
In addition to characterizing all compact Haussdorff, and Stone MV-algebras, we have completely characterized the MV-algebras that are isomorphic to their own profinite completions. An important related question (which is a weakening of strong completenes) is about (profinite) MV-algebras that are profinite completions of some MV-algebras. It is known \cite[Corollary 5.10]{BM} that every profinite Boolean algebra is a profinite completion of some Boolean algebra. It is also known that a profinite bounded distributive lattices that are isomorphic to profinite completions are up to isomorphism the lattices of upsets of representable posets \cite[Theorem 5.3]{BM}. The same question remains open for Heyting algebras (see comments after the proof of \cite[Corollary 5.9]{BM}). It is our goal to tackle this question for MV-algebras in an upcoming project.


\begin{thebibliography}{aaaaa}
 \bibitem{IGG} M. A. Alekselev, L. Yu. Glebskii, E. I. Gordon, On Approximation od groups, group actions, and Hopf algebras, J. Math. Sci, {\bf 107} , no. 5(2001)4305-4332.
 \bibitem{LA}  Lee W. Anderson, One dimensional topological lattices, Proc. Amer. Math. Soc. {\bf 10} (1959)715-720.
 \bibitem{BNS} L. P. Belluce, A. Di Nola, S. Sessa, The Prime Spectrum of an MV-algebra, Math. Log. Qart. {\bf 40} (1994)331-346.
\bibitem{GG} G. Bezhanishvili, N. Bezhanishvili, Profinite Heyting algebras,  \textit{Order} {\bf 25}(2008) 211-227.
 \bibitem{BH1} G. Bezhanishvili, J. Harding, On the proof that compact Hausdorff Boolean algebras are powersets, preprint.
  \bibitem{BH2} G. Bezhanishvili, J. Harding, Compact Hausdorff Heyting algebras, preprint.
  \bibitem{BM} G. Bezhanishvili, P. J. Morandi, Profinite Heyting algebras and profinite completions of Heyting algebras, Georgian Math. J. {\bf 16} (2009)29-47.
\bibitem{CG} T. H. Choe, R. J. Greechie, Profinite orthomodular lattices, Proc. Amer. Math. Soc. {\bf 118}, no. 4(1993)1053-1060.  
 \bibitem{Ci} R. Cignoli, Complete and atomic algebras of the infinite valued \L ukasiewicz logic, Studia Logica {\bf 50}, no. 3-4(1991) 375-384.
 \bibitem{C2} R. Cignoli, I. D'Ottaviano, D. Mundici, Algebraic foundations of many-valued reasoning,
 \textit{Kluwer Academic, Dordrecht}(2000).
 \bibitem{CM} R. Cignoli, D. Mundici, Stone duality for Dedekind $\sigma$-complete $\ell$-groups with order-unit, J. of Algebra, {\bf 302} (2006)848-861
 \bibitem{Hoo} C. S. Hoo, Topological MV-algebras, Topology Appl. {\bf 81}(1997)103-121.
 \bibitem{Jo} P. T. Johnstone, Stone spaces, Cambridge University Press, Cambridge (1982).
 \bibitem{Mu} D. Mundici, Advanced \L ukasiewicz calculus and MV-algebras, \textit{Trends Log. Stud. Log. Libr.}, {\bf 35}, Springer, New York (2011).
\bibitem{jbn} J. B. Nganou, Profinite MV-algebras and multisets, Order (2015) DOI 10.1007/s11083-014-9345-5.
\bibitem{DPS} D. P. Strauss, Topological lattices, Proc. London Math. Soc. (3)18(1968)217-230.
  \end{thebibliography}
\end{document}